\documentclass[twocolumn,final]{autart}

\usepackage{graphicx}
\usepackage{amsmath}
\usepackage{amssymb}

\usepackage{dsfont}
\usepackage{tikz}
\usepackage{caption}
\usepackage{subcaption}

\usepackage{xcolor}

\definecolor{cof}{RGB}{219,144,71}
\definecolor{pur}{RGB}{186,146,162}
\definecolor{greeo}{RGB}{91,173,69}
\definecolor{greet}{RGB}{52,111,72}

\newcommand{\bal}[1] {\ensuremath{\left(\begin{array}{#1}}}
\newcommand{\ear} {\ensuremath{\end{array}\right)}}

\newcommand{\bals}[1] {\ensuremath{\left[\begin{array}{#1}}} 
\newcommand{\ears} {\ensuremath{\end{array} \right] }} 
\newcommand{\one} {\ensuremath{\mathds{1} }} 

\DeclareMathOperator{\spa}{span}

\DeclareMathOperator{\rank}{rank}

\DeclareMathOperator{\sign}{sign}

\DeclareMathOperator{\diag}{diag}

\newcommand{\funcRdR}{\ensuremath{{V}}}
\newcommand{\funcRdRd}{\ensuremath{{f}}}

\DeclareMathOperator*{\bigtimes}{\raisebox{-0.3ex}{\text{\Large$\times$}}}

\let\leq\leqslant
\let\geq\geqslant
\let\emptyset\varnothing

\newcommand{\calC}{\ensuremath{\mathcal{C}}}

\newcommand{\calE}{\ensuremath{\mathcal{E}}}
\newcommand{\calF}{\ensuremath{\mathcal{F}}}
\newcommand{\calG}{\ensuremath{\mathcal{G}}}

\newcommand{\calI}{\ensuremath{\mathcal{I}}}

\newcommand{\calL}{\ensuremath{\mathcal{L}}}

\newcommand{\calR}{\ensuremath{\mathcal{R}}}
\newcommand{\calS}{\ensuremath{\mathcal{S}}}
\newcommand{\calT}{\ensuremath{\mathcal{T}}}

\newcommand{\calX}{\ensuremath{\mathcal{X}}}

\newcommand{\bmat}{\begin{matrix}}
\newcommand{\emat}{\end{matrix}}
\newcommand{\bbm}{\begin{bmatrix}}
\newcommand{\ebm}{\end{bmatrix}}
\newcommand{\bpm}{\begin{pmatrix}}
\newcommand{\epm}{\end{pmatrix}}
\newcommand{\bse}{\begin{subequations}}
\newcommand{\ese}{\end{subequations}}
\newcommand{\beq}{\begin{equation}}
\newcommand{\eeq}{\end{equation}}
\newcommand{\ben}{\begin{enumerate}}
\newcommand{\een}{\end{enumerate}}
\newcommand{\beni}{\renewcommand{\labelenumi}{\roman{enumi}.}
\renewcommand{\theenumi}{\roman{enumi}}\begin{enumerate}}
\newcommand{\eeni}{\end{enumerate}\renewcommand{\labelenumi}{\arabic{enumi}.}
\renewcommand{\theenumi}{\arabic{enumi}}}
\newcommand{\bena}{\renewcommand{\labelenumi}{\alpha{enumi}.}
\renewcommand{\theenumi}{\alpha{enumi}}\begin{enumerate}}
\newcommand{\eena}{\end{enumerate}\renewcommand{\labelenumi}{\arabic{enumi}.}
\renewcommand{\theenumi}{\arabic{enumi}}}
\newcommand{\bit}{\begin{itemize}}
\newcommand{\eit}{\end{itemize}}

\newcommand{\R}{\ensuremath{\mathbb R}}

\newcommand{\T}{\ensuremath{\top}}

\begin{document}

\begin{frontmatter}

\title{Finite-time Consensus Protocols for Multi-dimensional Multi-agent Systems}

\thanks[footnoteinfo]{This paper was not presented at any IFAC 
meeting. Corresponding author Jieqiang Wei.}

\author[a1]{J. Wei}\ead{jieqiang@kth.se},
\author[a2]{B. Besselink}\ead{b.besselink@rug.nl},
\author[a1]{J. Wu}\ead{junfengw@kth.se},
\author[a1]{H. Sandberg}\ead{hsan@kth.se},
\author[a1]{K. H. Johansson}\ead{kallej@kth.se}

\address[a1]{ACCESS Linnaeus Centre, School of Electrical Engineering, KTH Royal Institute of Technology, 10044, Stockholm, Sweden}  
\address[a2]{Johann Bernoulli Institute for Mathematics and Comp. Science, Univ. of Groningen, P.O. Box 407, 9700 AK, the Netherlands}             


\begin{abstract}\label{s:Abstract}
Two finite-time consensus protocols are proposed for multi-dimensional multi-agent systems, using direction-preserving and component-wise signum controls respectively. Filippov solutions and non-smooth analysis techniques are adopted to handle discontinuities. Sufficient and necessary conditions are provided to guarantee finite-time convergence and boundedness of the solutions. It turns out that the number of agents which have continuous control law plays an essential role for finite-time convergence. In addition it is shown that the unit balls introduced by $\ell_p$ and $\ell_\infty$ norms are invariant for these two protocols respectively. 
\end{abstract}

\end{frontmatter}


\section{Introduction}\label{s:Introduction}
Multi-agent systems have a broad spectrum of applications both in military and civilian environments, and have been a focus area of research for decades. The essential goal of the control of multi-agent systems is to let the agents achieve some state cooperatively with only local information exchange, e.g., \cite{Jadbabaie2003,Olfati-Saber2007}.
Whereas most results on control of multi-agent systems focus on asymptotic convergence properties, we study a \emph{finite-time} convergence problem in this paper.

Existing results on finite-time convergence of multi-agent systems can roughly be divided into two groups: namely those exploiting continuous or discontinuous control protocols. These groups have in common that the nonlinear control laws are not Lipschitz continuous at the desired consensus space. For example, continuous strategies are usually based on applying a nonlinear state feedback strategy that includes fractional powers, e.g., \cite{Hui2008,wangxiao2010,Xiao2009}, or on a high gain converging to infinity as time approaches the converging time, e.g., \cite{Harl2012}.
The discontinuous strategies utilize non-smooth control tools, e.g., \cite{ChenCaoRen2012,Chen2011,Cortes2006,Hui2010,Li2014,LiuLam2016}, typically exploiting signum functions.
%
%
In \cite{Chen2011}, the authors construct a finite-time consensus law using binary information, namely, the sign of state differences of each pair of agents. A different approach towards finite-time consensus is taken in \cite{Cortes2006}, where controllers are studied that contain the sign of the sum of the state differences. However, in this work, the boundedness of (Filippov) solutions can not be guaranteed.
Sufficient conditions for boundedness of Filippov solutions are given in \cite{Wei2015} for a general class of nonlinear multi-agent systems that includes the results in \cite{Cortes2006} as a special case. However, in \cite{Wei2015}, only asymptotic convergence properties are considered and no results on finite-time consensus are available. In the current paper, such finite-time convergence properties are studied.


Next, we note that the existing results on finite-time consensus mentioned above typically hold for multi-agent systems in which the agent dynamics is scalar. Nonetheless, there are many applications in which multi-dimensional agent dynamics are of interest, see, e.g., the problems of attitude control~\cite{Thunberg2014} and circumnavigation~\cite{Shames2012}. The current paper therefore focuses on multi-dimensional agents.

Specifically, the contributions of this paper are as follows. First, a control strategy for multi-dimensional multi-agent systems is presented that relies on the definition of a \emph{direction-preserving} signum function. For any $p\in[0,\infty]$, the corresponding direction-preserving signum function maps nonzero points to the boundary of the unit ball in $\ell_p$, but such that the direction is preserved. As a result, the analysis of multi-agent systems subject to this control strategy differs significantly from the scalar case. Second, for these systems, it is shown that the unit ball in the $\ell_p$ norm is invariant under the dynamics of the multi-agent systems. Third, necessary and sufficient conditions for the finite-time consensus of such multi-agent systems are given. These conditions indicate that, in order to achieve finite-time \emph{static} consensus, one of the agents should employ a Lipschitz continuous control strategy. Finally, it is shown how the results of this paper can be used to study multidimensional multi-agent systems with \emph{component-wise} signum functions.


The structure of the paper is as follows. In Section~\ref{s:Preliminaries}, we introduce terminology and notation on graph theory and stability analysis of discontinuous dynamical systems. Section~\ref{s:problem formulation} presents the problem formulation of finite-time consensus. The main result is presented in Section~\ref{s:vector}, which includes illustrative examples. Then the conclusions follow in Section~\ref{s:conclusions}.

\textbf{Notation}. With $\R_-,\R_+, \R_{\geq 0}$ and $\R_{\leq 0}$ we denote the sets of negative, positive, non-negative, non-positive  real numbers, respectively. A positive semidefinite (symmetric) matrix $M$ is denoted as $M\succcurlyeq 0$. The $i$-th row of a matrix $M$ is given by $M_{i}$. The vectors $e_1,e_2,\ldots,e_n$ denote the canonical basis of $\R^n$, whereas the vectors $\one_n$ and $\mathbf{0}_n$ represent a $n$-dimensional column vector with each entry being $1$ and $0$, respectively. We will omit the subscript $n$ when no confusion arises. The $\ell_p$ norm with $p\in[1,\infty]$ is denoted as $\|\cdot\|_p$. Specifically, for a vector $x\in\R^n$, $\|x\|_p=(|x_1|^p+\cdot+|x_n|^p)^{\frac{1}{p}}$ for $p\in[1,\infty)$, and $\|x\|_\infty= \max\{|x_1|,\ldots,|x_n|\}$. The notation $B(x,\delta)$ represents the open ball centered at $x$ with radius $\delta>0$ with $\ell_2$ norm.


\section{Preliminaries}\label{s:Preliminaries}
In this section we briefly review some essentials from graph theory \cite{biggs1993algebraic,Bollobas98}, and give some results on Filippov solutions \cite{filippov1988} of differential equations with discontinuous vector fields.

An undirected \emph{graph} $\calG=(\calI,\calE)$ consists of a finite set of nodes $\calI = \{1,2,\ldots,n\}$ and a set of edges $\calE\in\calI\times\calI$ of unordered pairs of elements of $\calI$. To any edge $(i,j)\in\calE$, we associate a weight $w_{ij}>0$. The weighted adjacency matrix $A = [a_{ij}]\in\R^{n\times n}$ is defined by $a_{ji} = w_{ij}$ if $(i,j)\in\calE$ and $a_{ji} = 0$ otherwise. Note that $A = A^{\T}$ and that $a_{ii}=0$ as no self-loops are allowed. For each node $i\in\calI$, its degree $d_i$ is defined as $d_i = \sum_{j=1}^n a_{ij}$. The graph Laplacian $L$ is defined as $L = \Delta - A$ with $\Delta$ a diagonal matrix such that $\Delta_{ii}=d_i$. As a result, $L\one = \mathbf{0}$. Finally, we say that a graph $\calG$ is connected if, for any two nodes $i$ and $j$, there exists a sequence of edges that connects them. In order to simplify the notation in the proofs, we set the weight $w_{ij}$ to be one. All the results in this paper hold for general positive nonzero $w_{ij}$.

The following result essentially states that the Schur complement of a graph Laplacian is itself a graph Laplacian.
\begin{lem}[\cite{vanderSchaft2010}]\label{lem_schur complement}
Consider a connected undirected graph $\calG$ with Laplacian matrix $L$, then all Schur complements of $L$ are well-defined, symmetric, positive semi-definite, with diagonal elements $>0$, off-diagonal elements $\leq 0$, and with zero row and column sums.
\end{lem}

In the remainder of this section we discuss Filippov solutions. Let $\funcRdRd$ be a map from $\R^n$ to $\R^n$ and let $2^{\R^n}$ denote the collection of all subsets of $\R^n$. Then, the \emph{Filippov set-valued map} of $\funcRdRd$, denoted $\calF[\funcRdRd]:\R^n\rightarrow 2^{\R^n}$, is defined as
\begin{equation}
\calF[\funcRdRd](x) \triangleq \bigcap_{\delta>0}\bigcap_{\mu(S)=0}\overline{\mathrm{co}}\big\{ \funcRdRd(B(x,\delta)\backslash S) \big\},
\label{eqn_Filippovdef}
\end{equation}
where $S$ is a subset of $\R^n$, $\mu$ denotes the Lebesgue measure and $\overline{\mathrm{co}}\{\calX\}$ denotes the convex closure of a set $\calX$. If $\funcRdRd$ is continuous at $x$, then $ \calF[\funcRdRd](x)$ contains only the point $\funcRdRd(x)$.

Properties of the Filippov set-valued map are stated next.
\begin{prop}[Calculus for $\calF$ \cite{paden1987}]\label{p:calculus for Filippov}
The following properties hold for the Filippov set-valued map (\ref{eqn_Filippovdef}):
	\begin{enumerate}
		\item Assume that $f_j:\R^m\rightarrow \R^{n_j}$, $j=1,\ldots,N$ are locally bounded, then
		\begin{equation}
		\calF\bigg[ \bigtimes_{j=1}^N f_j \bigg](x) \subset \bigtimes_{j=1}^N\calF[f_j](x). \footnote{Cartesian product notation and column vector notation are used interchangeably.}
		\end{equation}
		\item Let $g:\R^m\rightarrow\R^n$ be $C^1$, $\rank Dg(x)=n$ and\footnote{A function is of class $C^1$ if it is continuously differentiable; $Df$ denotes the Jacobian of $f$.} $f:\R^n\rightarrow\R^p$ be locally bounded; then
		\begin{equation}
		\calF[f\circ g](x)=\calF[f](g(x)). 
		\end{equation}
	\end{enumerate}
\end{prop}

A \emph{Filippov solution} of the differential equation $\dot{x}=\funcRdRd(x)$ on $[0,T]\subset\R$ is an absolutely continuous function $x:[0,T]\rightarrow\R^n$ that satisfies the differential inclusion
\begin{equation}\label{e:differential_inclusion}
\dot{x}(t)\in \calF[\funcRdRd](x(t))
\end{equation}
for almost all $t\in[0,T]$. A Filippov solution $t\mapsto x(t)$ is \emph{maximal} if it cannot be extended forward in time, that is, if $t\mapsto x(t)$ is not the result of the truncation of another solution with a larger interval of definition. Since Filippov solutions are not necessarily unique, we need to specify two types of invariant set. A set $\calR\subset\R^n$ is called \emph{weakly invariant} if, for each $x_0\in \calR$, at least one maximal solution of \eqref{e:differential_inclusion} with initial condition $x_0$ is contained in $\calR$. Similarly, $\calR\subset \R^n$ is called \emph{strongly invariant} if, for each $x_0\in \calR$, every maximal solution of \eqref{e:differential_inclusion} with initial condition $x_0$ is contained in $\calR$. For more details, see \cite{cortes2008,filippov1988}.

If $V:\R^n\rightarrow\R$ is locally Lipschitz, then its \emph{generalized gradient} $\partial V:\R^n\rightarrow 2^{\R^n}$ is defined by
\begin{equation}
\partial V(x):=\mathrm{co}\Big\{\lim_{i\rightarrow\infty} \nabla
V(x_i):x_i\rightarrow x, x_i\notin S\cup \Omega_{\funcRdR} \Big\},
\end{equation}
where $\mathrm{co}\{\calX\}$ denotes the convex hull of a set $\calX$, $\nabla$ denotes the gradient operator, $\Omega_{\funcRdR} \subset\R^n$ denotes the set of points where $V$ fails to be differentiable and $S\subset\R^n$ is a set of measure zero that can be
arbitrarily chosen to simplify the computation. Namely, the resulting set $\partial V(x)$ is independent of the choice of $S$ \cite{Clarke1990optimization}.

Given a  set-valued map $\calT:\R^n\rightarrow 2^{\R^n}$, the \emph{set-valued Lie derivative} $\calL_{\calT}V:\R^n\rightarrow 2^{\R^n}$ of a locally Lipschitz function $V:\R^n\rightarrow \R$  with respect to $\calT$ at $x$ is defined as
\begin{equation}\label{e:set-valuedLie}
\begin{aligned}
\calL_{\calT}V(x) := & \big\{ a\in\R \mid \exists\nu\in\calT(x) \textnormal{ such that } \\
& \quad \zeta^T\nu=a,\, \forall \zeta\in \partial V(x)\big\}.
\end{aligned}
\end{equation}
If $\calT(x)$ is convex and compact $\forall x\in\R^n$, then $\calL_{\calT}V(x)$ is a closed and bounded interval in $\R$, possibly empty, for each $x$.

The following result is a generalization of LaSalle's invariance principle for discontinuous differential equations \eqref{e:differential_inclusion} with non-smooth Lyapunov functions.
\begin{thm}[LaSalle's Invariance Principle \cite{Cortes2006}]\label{chap_preli:thm_stability}
Let $V:\R^n\rightarrow\R$ be a locally Lipschitz and regular function\footnote{The definition of a regular function can be found in \cite{Clarke1990optimization} and we recall that any convex function is regular.}. Let $S\subset \R^n$ be compact and strongly invariant for \eqref{e:differential_inclusion} and assume that $\max \calL_{\calF[f]} V(x)\leq 0$ for all $x\in S$, where we define $\max\emptyset=-\infty$. Let
\begin{equation}
Z_{\calF[f],V}= \big\{ x\in\R^n \;\big|\; 0\in\calL_{\mathcal{F}[f]}V(x) \big\}.
\end{equation}
Then, all solutions $x:[0,\infty)\rightarrow \R^n$ of \eqref{e:differential_inclusion} with $x(0)\in S$ converge to the largest weakly invariant set $M$ contained in
\begin{equation}
S\cap\overline{Z_{\calF[f],V}}.
\end{equation}
Moreover, if $M$ consists of a finite number of points, then the limit of each solution starting in $S$ exists and is an element of~$M$.
\end{thm}
A result on finite-time convergence for \eqref{e:differential_inclusion} is stated next, which will form the basis for our results on finite-time consensus for multi-agent systems.
\begin{lem}[\cite{Cortes2006}]\label{p:finite_time_convergence}
Under the same assumptions as in Theorem \ref{chap_preli:thm_stability}, if $\max \calL_{\calF[f]} V(x)<\varepsilon<0$ a.e.\ on $S\setminus Z_{\calF[f],V}$, then $Z_{\calF[f],V}$ is attained in finite time.
\end{lem}

\section{Problem formulation}\label{s:problem formulation}
Consider the nonlinear multi-agent system
\begin{equation}\label{e:plant}
\dot x_i = u_i, \quad i\in\calI =\{1,2,\ldots,n\},
\end{equation}
defined on a connected network $\calG$ with $n$ agents and $m$ edges, i.e., $|\calE|=m$, where $x_i(t),\, u_i(t) \in \R^k$ are the state and the input of agent $i$ at time $t$, respectively. 

By defining the consensus space as
\begin{align}
\calC = \big\{ x\in\R^{kn} \mid \exists \bar{x}\in\R^k \textnormal{ such that } x = \mathds{1}\otimes \bar{x} \big\},
\end{align}
we say the states converge to consensus in finite time if for any initial condition there exists a time $t^*>0$ such that $x=[x_1,\ldots,x_n]^\top$ converge to a \emph{static} vector in $\calC$ as $t\rightarrow t^*$, i.e., there exists a vector $\bar{x}$ such that $x(t) = \mathds{1}\otimes\bar{x}$ for all $t\geq t^*$. In this paper, we design control inputs $u_i$ such that the states of the system \eqref{e:plant} converge to consensus in finite time. We understand the trajectories of the system in the sense of Filippov. We formally formulate our objective as follows. 

\textbf{Aim \ }Design control protocols to the system \eqref{e:plant} such that all Filippov solutions converge to a \emph{static} vector in $\calC$ in finite time.

\section{Main results}\label{s:vector}
For the design of the control input $u_i$, we start with a general nonlinear form 
\begin{equation}\label{e:control}
u_i = f_i\!\Bigg(\sum_{j=1}^n a_{ij}(x_j-x_i)\Bigg), \quad i\in\calI,
\end{equation}
where $f_i:\R^k\to \R^k$ and $a_{ij}$ is the $ij$-th elements of the adjacency matrix $A$. 
Combining \eqref{e:plant} and \eqref{e:control}, the closed-loop system is given as  
\begin{equation}\label{e:generalsystem}
\dot x_i = f_i\!\Bigg(\sum_{j=1}^n a_{ij}(x_j-x_i)\Bigg), \quad i\in\calI.
\end{equation}
Denoting $\bar{L}=L\otimes I_k$ and $\bar{L}_i=L_{i,\cdot}\otimes I_k$, \eqref{e:generalsystem} can be further written in a compact form as
\begin{equation}\label{e:nonlinear_vector}
\dot{x}=f(-\bar{L}x)
\end{equation}
in which $x=[x^\top_1, x^\top_2, \ldots, x^\top_n]^\top\in\R^{kn}$ collects the states of all agents and $f(y) = [f_1^\top(y_1), \ldots, f_n^\top(y_n)]^\top$.

In light of the success of scalar binary control protocols to achieve finite-time consensus \cite{Chen2011,Cortes2006,Hui2010}, we shall design control protocols based on the signum function for multi-dimensional multi-agent systems.

We consider both the direction-preserving signum \eqref{e:sign} and component-wise signum \eqref{e:signc} functions. Specifically, we denote the direction-preserving signum function as
\begin{equation}\label{e:sign}
\sign(w) = \begin{cases}
\frac{w}{\|w\|_p} & \textrm{ if } w\neq \mathbf{0},\\
0 & \textrm{ if } w=\mathbf{0},
\end{cases}
\end{equation}
for $w\in\R^k$ and any $p\in[1,\infty]$. Moreover, the component-wise signum function is given as
\begin{equation}\label{e:signc}
\sign_c(w)=[\begin{array}{ccc}\sign(w_1) & \cdots & \sign(w_k)\end{array}]^{\T}. 
\end{equation}
Notice that, for $k=1$, these two signum functions coincide. Furthermore, the component-wise signum is coarser than its direction-preserving counterpart in the sense that there is only a finite number of elements in the range of $\sign_c$ for a fixed dimension $k$. A graphical comparison between these two functions can be found in Fig.~\ref{fig:sign}.

In the following two subsections, we propose control protocols based on \eqref{e:sign} and \eqref{e:signc}, respectively, and derive convergence results. 

\begin{figure}
\centering
\begin{tikzpicture}[scale=1.2]
    \draw [<->,thick] (0,1.5) node (yaxis) [above] {$y$}
        |- (2,0) node (xaxis) [right] {$x$};
    \draw [->, thick] (0,0)  -- (1,1)  node (sign) [above] {$\sign_c(\phi)$};
    \draw [->, thick] (0,0)  -- (1.5,0.5)  node (phi) [right] {$\phi$};
    \draw [dashed] (0,1) node (01) [left] {$(0,1)$} -- (1,1) coordinate (b_2);
    \draw [dashed] (1,0) node (10) [below] {$(1,0)$} -- (1,1) coordinate (b_2);
\end{tikzpicture}
\begin{tikzpicture}[scale=1.2]
    \draw [<->,thick] (0,1.5) node (yaxis) [above] {$y$}
        |- (2,0) node (xaxis) [right] {$x$};
    \draw [->, thick] (0,0)  -- (1.5,0.5)  node (phi) [right] {$\phi$};
    \draw [dashed] (1,0) node (10) [below] {$(1,0)$} arc (0:90:1) (0,1) node (01) [left] {$(0,1)$};
    \draw [->, thick] (0,0)  -- (0.9487,0.3162)  node (sign) [left] {\tiny{$\sign(\phi)$}};
\end{tikzpicture}
\caption{The difference between component-wise signum $\sign_c$ and the direction-preserving signum \eqref{e:sign} with $\ell_2$-norm in $\R^2$. Left: the vector $\sign_c(\phi)=(1,1)$ is the component-wise sign of $\phi$ in the first quadrant. Right: the vectorized sign of $\phi$ reside on the unit circle.}\label{fig:sign}
\end{figure}
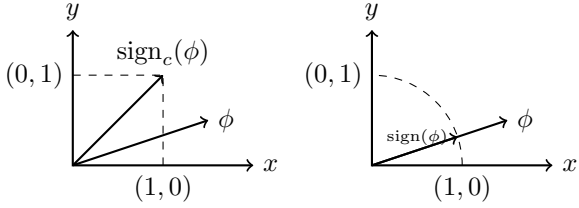

\subsection{Direction-preserving signum}\label{s:direction-preserving sign}
In this subsection, we consider the nonlinear controller $u_i$ as in \eqref{e:control} with some  $f_i=\sign$ as in \eqref{e:sign}. One intuitive idea is to employ the controller 
\begin{equation}\label{e:control_false}
u_i = \sign\Bigg(\sum_{j=1}^n a_{ij}(x_j-x_i)\Bigg), \quad i\in\calI,
\end{equation}
such that all the functions $f_i$ in \eqref{e:control} are signum functions. However, as pointed out in~\cite{Wei2015}, if $x_i\in\R$ and $f_i=\sign$ for all $i\in\calI$, consensus in the sense of this paper is not possible for \eqref{e:nonlinear_vector}, even asymptotically. Indeed, the trajectories will not converge to a \emph{static} vector in $\calC$. We explain this phenomenon by recalling an example from \cite{Wei2015}, which also serves as an counterexample to the result in Section~4 of \cite{Cortes2006}.  

\begin{exmp}\cite{Wei2015}\label{example_undi_3nodes}
Consider
\begin{align*}
\dot{x}_1 &= \sign(x_2+x_3-2x_1), \\
\dot{x}_2 &= \sign(x_1+x_3-2x_2), \\
\dot{x}_3 &= \sign(x_1+x_2-2x_3),
\end{align*}
defined on a circular graph with three nodes and all edges with unit weight, and $x_i\in\R$. In this case $\calC= \spa\{\one_3\}$. We can show that for any initial condition $x(0)\notin\calC$, all Filippov solutions converge to $\calC$ in finite time. However, once they enter $\calC$, the solutions can be unbounded. Indeed, suppose that at time $t_0$ we have $x(t_0) \in\calC$, then
\begin{equation}\label{e:filippovset-example}
\calF[h](x(t_0))=\overline{\mathrm{co}}\left\{
\nu_1,\nu_2,\nu_3,-\nu_1,-\nu_2,-\nu_3 \right\},
\end{equation}
where $\nu_1=[1,1,-1]^T$, $\nu_2=[1,-1,1]^T$, and $\nu_3=[-1,1,1]^T$. Since
$\sum_{i=1}^3 \frac{1}{3} \nu_i =\frac{1}{3}\one$, we have that
\begin{align}
\big\{\eta\one \;\big|\; \eta\in[-\tfrac{1}{3},\tfrac{1}{3}]\big\} \subset \calF[h](x(t_0)).
\end{align}
Hence, any function $x(t)=\eta(t)\one$ with $\eta(t)$ differentiable almost everywhere and satisfying $\dot{\eta}(t)\in[-\tfrac{1}{3},\tfrac{1}{3}]$ for all $t>t_0$ is a Filippov solution for this system. Consequently, not all Filippov solutions converge to a static vector in $\calC$.
\end{exmp}

For the scalar version of \eqref{e:plant} with controller \eqref{e:control_false}, Theorem~7 in \cite{Wei2015} provides a guarantee for \emph{asymptotic} convergence to consensus. The main result in Theorem~7 \cite{Wei2015} relies on replacing one $\sign$ function with a function that is Lipschitz-continuous at the origin. This result prompts us to the following assumption on $f_i, i\in\calI$.

\begin{assum}\label{ass_vectorfield_1}
For some set $\calI_c\subset\calI$, the function $f$ in~(\ref{e:nonlinear_vector}) satisfies the following conditions:
	\begin{enumerate}
		\item[\textit{(i)}] For $i\in\calI_c$, the function $f_i:\R^k\rightarrow\R^k$ is locally Lipschitz continuous and satisfies $f_i(\mathbf{0}) = \mathbf{0}$ and $f_i(y)^\top y=\|f_i(y)\|\cdot\|y\|> 0$ for all $y\neq \mathbf{0}$ (i.e., the functions $f_i$ are direction preserving); 
		\item[\textit{(ii)}] For $i\in\calI\backslash\calI_c$, the function $f_i$ is the direction-preserving signum, i.e., $f_i = \sign$.
	\end{enumerate}
\end{assum}

\begin{rem}
For the scalar case, \emph{direction preserving} in Assumption \ref{ass_vectorfield_1} is simply \emph{sign preserving}, i.e., $f_i(0) = 0$ and $f_i(y_i) y_i> 0$ for all $y_i\neq 0$. 
\end{rem}
Notice that the $\sign$ function in (\ref{e:sign}) is locally Lipschitz and direction preserving on $\R^k\setminus\{\mathbf{0}\}$, and for any $w\neq \mathbf{0}$ we have $\|\sign(w)\|_p=1$. 

To handle the discontinuities in $f$ that arise from the signum function in Assumption~\ref{ass_vectorfield_1}, we understand the solution of \eqref{e:nonlinear_vector} in the sense of Filippov, i.e., we consider the differential inclusion
\begin{equation}\label{e:nonlinear_vector_fili}
\dot{x}\in \calF[h](x),
\end{equation}
where $h=f(-\bar{L}x)$.

\medskip

So far, one could expect that the finite-time convergence of system~(\ref{e:nonlinear_vector_fili}) to consensus hinges upon the conditions of $\calI$ and~$\calI_c$. We start the analysis with the special case with $\calI_c = \calI$. It is well-known that Lipschitz continuous vector fields give mere asymptotic convergence, but this result is stated and proven explicitly for completeness. 
\begin{prop}\label{prop: asymptotic_convergence}
Consider the nonlinear consensus protocol \eqref{e:nonlinear_vector} satisfying Assumption~\ref{ass_vectorfield_1} with $\calI_c = \calI$. Then, the (unique) solution of \eqref{e:nonlinear_vector}
converges to consensus only asymptotically, i.e., finite-time consensus is not achieved.
\end{prop}
\begin{proof}
Before setting up an argument by contradiction, note that the (local) Lipschitz continuity of $f_i$ implies that $h(x) = f(-\bar{L}x)$ is also locally Lipschitz continuous. As a result of the Picard-Lindel\"{o}f theorem, this implies that, for a given initial condition, the solution of (\ref{e:nonlinear_vector}) is unique forward and backward in time.

Now, in order to establish a contradiction, let $x(t)$ be a solution, with $x(0)\notin\calC$, of (\ref{e:nonlinear_vector}) that achieves consensus in finite time, i.e., there exists a time $t_0<\infty$ such that $x(t_0) = \one\otimes\eta$ for some vector $\eta\in\R^k$ and $x(t) \neq \one\otimes\eta$ for all $t<t_0$. Now, by the observation that $h(\one\otimes\eta) = \mathbf{0}$, it follows that $\bar{x}(t) = \one\otimes\eta$ for all $t$ is also a solution of (\ref{e:nonlinear_vector}), contradicting the uniqueness of solutions. Consequently, for $\calI_c = \calI$, (\ref{e:nonlinear_vector}) will not achieve consensus in finite time. 
\end{proof}

One key property of the system \eqref{e:plant} with controller \eqref{e:control} satisfying Assumption \ref{ass_vectorfield_1} is that any bounded ball in the $\ell_p$-norm is strongly invariant. This is formulated in the following lemma.

\begin{lem}\label{lm:invariant_vec}
Consider the differential inclusion (\ref{e:nonlinear_vector_fili}) satisfying Assumption \ref{ass_vectorfield_1} for some $p\in[0,\infty]$. If one of the following two conditions is satisfied
\begin{enumerate}
\item[\textit{(i)}] $|\calI| = 2$ and $|\calI_c| = 0$;
\item[\textit{(ii)}] $|\calI| \geq 2$ and $|\calI_c| \geq 1$,
\end{enumerate}
then the set $\calS_p(C) = \{ x\in\R^{nk}\mid \|x_i\|_p\leq C, i\in\calI \}$, where $C>0$ is a constant, is strongly invariant.
\end{lem}

\begin{proof}
We divide the proof into two parts, discussing the cases $p\in [1,\infty)$ and $p=\infty$ separately.

\textit{(1)}. Let $p\in [1,\infty)$. We introduce a Lyapunov function candidate
\begin{align}
V(x) = \max_{i\in\calI} \frac{1}{p}\|x_i\|_p^p
\label{eqn_prf_lm:invariant_V}
\end{align}
and note that $V(x)\leq \frac{1}{p}C^p$ implies that $x\in\calS_p(C)$. Since the function $(\cdot)^p$ is convex on $\R_{\geq 0}$, it can be observed that $V$ is convex and, hence, regular. In the remainder of the proof, we will show that $V(x(t))$ is non-increasing along all Filippov solutions of (\ref{e:nonlinear_vector_fili}), implying strong invariance of the set $\calS_p(C)$ for any $C>0$.

Let $\alpha(x)$ denote the set of indices that achieve the maximum in (\ref{eqn_prf_lm:invariant_V}) as
\begin{align}
\alpha(x) = \big\{ i\in\calI \mid \|x_i\|_p^p = pV(x) \big\}.
\label{eqn_prf_lm:invariant_alpha}
\end{align}	
Then, the generalized gradient of $V$ in (\ref{eqn_prf_lm:invariant_V}) is given by
\begin{equation}
\partial V(x)=\mathrm{co} \big\{ e_i\otimes \psi(x_i) \mid i\in\alpha(x) \big\}
\label{eqn_prf_lm:invariant_gengrad}
\end{equation}
where 
\begin{align}\label{e:psi}
\psi(x_i) = \left[\begin{array}{c}
|x_{i,1}|^{p-1}\calF[\sign](x_{i,1}) \\ \vdots \\
|x_{i,k}|^{p-1}\calF[\sign](x_{i,k})\end{array}\right]
\end{align}
and $x_i = [\begin{array}{ccc} x_{i,1} & \ldots & x_{i,k}\end{array}]^{\T}\in\R^k$.

Next, let $\Psi$ be defined as
\begin{equation}
\Psi = \big\{ t\geq 0 \mid  \dot{x}(t) \textnormal{ and } \tfrac{d}{dt}V(x(t)) \textnormal{ exist} \big\}.
\end{equation}
Since $x$ is absolutely continuous (by definition of Filippov solutions) and $V$ is locally Lipschitz, by Lemma~1 in~\cite{Bacciotti1999} it follows that $\Psi=\R_{\geq 0}\setminus\bar{\Psi}$ for a set $\bar{\Psi}$ of measure zero and
\begin{equation}
\frac{d}{dt}V(x(t))\in \calL_{\calF[h]}V(x(t))
\end{equation}
for all $t\in\Psi$, such that the set $\calL_{\calF[h]}V(x(t))$ is nonempty for all $t\in\Psi$. For $t\in\bar{\Psi}$, we have that $\calL_{\calF[h]}V(x(t))$ is empty, and hence $\max \calL_{\calF[h]}V(x(t)) = -\infty < 0$ by definition. Therefore, we only consider $t\in\Psi$ in the rest of the proof.
	
Next, we will consider the cases $x\in\calC$ and $x\notin\calC$ separately.

First, for $x\in\calC$, it can be observed that $\alpha(x) = \calI$. Then, the following two cases can be distinguished.
\begin{enumerate}
	\item[\textit{(i)}] $|\calI| \geq 2$ and $|\calI_c|\geq 1$. As there is at least one agent with continuous vector field, there exists $i\in\calI$ such that $f_i$ is locally Lipschitz and direction preserving. Then, by definition of the Filippov set-valued map, it follows that $\nu_i = 0$ for all $\nu\in\calF[h](x)$ (recall $x\in\calC$). As $\calL_{\calF[h]}V(x(t))$ is nonempty (by considering $t\in\Psi$), there exists $a\in\calL_{\calF[h]}V(x(t))$ such that $a = \zeta^\top\nu$ for all $\zeta\in\partial V(x(t))$, see the definition (\ref{e:set-valuedLie}). Choosing $\zeta = e_i\otimes \psi(x_i(t))$, it follows that $a = (e_i\otimes \psi(x_i(t)))^\top\nu = 0$, which implies that $\max\calL_{\calF[h]}V(x(t)) = 0 \leq 0$, i.e., $V(x)$ is non-increasing for any $x\in\calC$.
	
	\item[\textit{(ii)}] $|\calI| = 2$ and $|\calI_c| = 0$. In this case, the system (\ref{e:nonlinear_vector}) can be written as
	\begin{equation}
	\begin{aligned}
	\dot{x}_1 & = \frac{x_2-x_1}{\|x_2-x_1\|_p}, \\
	\dot{x}_2 & = \frac{x_1-x_2}{\|x_1-x_2\|_p}.
	\end{aligned}\label{eqn_prf_lm:invariant_twoagents}
	\end{equation}
	Then, by using the definition (\ref{eqn_Filippovdef}), it can be shown that, for $x_1 = x_2$ (i.e., $x\in\calC$), any element $\nu$ in the Filippov set-valued map of (\ref{eqn_prf_lm:invariant_twoagents}) satisfies $\nu_1 = -\nu_2$. Stated differently, the following implication holds with $\nu = [\nu_1^\top,\nu_2^\top]^\top$:
	\begin{align}
	\nu\in\calF[h](x),\; x\in\calC \;\Rightarrow\; \nu_1 = -\nu_2.
	\end{align}
	Next, by recalling that $\alpha(x) = \calI$ (see (\ref{eqn_prf_lm:invariant_alpha})), it follows from (\ref{eqn_prf_lm:invariant_gengrad}) that
	\begin{align}
	\partial V(x) = \mathrm{co}\big\{e_1\otimes \psi(x_1(t)), e_2\otimes \psi(x_2(t))\big\}
	\end{align}
	with $x_1 = x_2$. Now, following a similar reasoning as in item (i) on the basis of the definition of the set-valued Lie derivative in (\ref{e:set-valuedLie}), it can be concluded that $a = \zeta^\top\nu$ is necessarily $0$, such that $\max\calL_{\calF[h]}V(x(t)) = 0 \leq 0$ for all $x\in\calC$.
\end{enumerate}

Second, the case $x\notin\calC$ is considered. For this case, Proposition~\ref{p:calculus for Filippov} is applied to obtain
\begin{align}
\calF[h](x) \subset \bigtimes_{j=1}^N \calF[f_i]\big(-\bar{L}_ix\big) =: \bar{\calF}(x),
\label{eqn_prf_lm:invariant_Finclusion}
\end{align}
after which it follows from the definition of the set-valued Lie derivative (\ref{e:set-valuedLie}) that
\begin{align}
\calL_{\calF}V(x) \subset \calL_{\bar{\calF}}V(x).
\label{eqn_prf_lm:invariant_LVinclusion}
\end{align}
Therefore, in the remainder of the proof for the case $x\notin\calC$, we will show that $\max\calL_{\bar{\calF}[h]}V(x(t)) \leq 0$, which implies the desired result by (\ref{eqn_prf_lm:invariant_LVinclusion}). As before, it is sufficient to consider the set $\Psi$ such that $\calL_{\bar{\calF}[h]}V(x(t))$ is non-empty for all $t\in\Psi$.

Now, take an index $i\in\alpha(x)$ such that $\bar{L}_ix \neq \mathbf{0}$. Note that such $i$ indeed exists. Namely, assume in order to establish a contradiction that $\bar{L}_ix = \mathbf{0}$ for all $i\in\alpha(x)$. If $\alpha(x)=\calI$, then there exists $\ell\in\{1,\ldots,k\}$ such that 
\begin{align}
\beta(\ell):=\arg\max_{j\in\calI}x_{j,\ell} \subsetneq \calI,
\end{align}
i.e., there exists a state component $\ell$ that does not have the same value for all agents. Otherwise, $x\in\calC$, which is a contradiction. Then for any $i\in\beta(\ell)$ with $j\in N_i\setminus \beta(\ell)$, we have $\bar{L}_ix \neq \mathbf{0}$, where $N_i$ is the set of neighbors of agent $i$. If $\alpha(x)\subsetneq \calI$ and $\bar{L}_ix = \mathbf{0}$ for all $i\in\alpha(x)$, then for any $i\in\alpha(x)$ with $j\in N_i\setminus\alpha(x)$ we have
\begin{align}
0 & = \psi(x_i)^\top\bar{L}_ix  \label{eqn_prf_lm:invariant_Lix-4}\\
& = \sum_{j\in N_i} \bigg(\|x_i\|_p^p - \sum_{\ell=1}^{k}|x_{i,\ell}|^{p-1}\calF[\sign](x_{i,\ell})x_{j,\ell}\bigg)\\
& \geq \sum_{j\in N_i} \bigg(\|x_i\|_p^p - \sum_{\ell=1}^{k}|x_{i,\ell}|^{p-1} |x_{j,\ell}|\bigg)\\
& \geq  \sum_{j\in N_i} \big(\|x_i\|_p^p - \|x_j\|_p \|x_i\|_p^{p-1} \big)
\label{eqn_prf_lm:invariant_Lix} \\
& > 0
\end{align}
where the inequality \eqref{eqn_prf_lm:invariant_Lix} is based on H\"{o}lder's inequality, and the last inequality is implied by $\|x_i\|>\|x_j\|$ for any $j\in N_i\setminus\alpha(x)$. This is a contradiction.

For the index $i\in\alpha(x)$ satisfying $\bar{L}_ix\neq \mathbf{0}$, it follows from Assumption~\ref{ass_vectorfield_1} that there exists $\gamma>0$ such that
\begin{align}
\calF[f_i]\big(-\bar{L}_ix\big) = \{-\gamma\bar{L}_ix \},
\end{align}
i.e., for any $\nu\in\bar{\calF}(x)$ it holds that $\nu_i = -\gamma\bar{L}_ix$.
Note that this is a result of the direction-preserving property of either the direction-preserving signum (for a nonzero argument, then $\gamma = \frac{1}{\|\bar{L}_ix\|_p}$) or the Lipschitz continuous function (by Assumption~\ref{ass_vectorfield_1}). Then, choosing $\zeta\in\partial V(x)$ as $\zeta = e_i\otimes \psi(x_i)$ (recall that $i\in\alpha(x)$), it follows from (\ref{e:set-valuedLie}) that
\begin{align}
\calL_{\bar{\calF}[h]}V(x) = \{-\gamma \psi(x_i)^\top\bar{L}_ix\}.
\label{eqn_prf_lm:invariant_LbarV}
\end{align}
Next, by observing \eqref{eqn_prf_lm:invariant_Lix} we have
\begin{align}
\psi(x_i)^\top\bar{L}_ix \geq 0, 
\end{align}
which implies $\calL_{\bar{\calF}[h]}V(x)\subset \R_{\leq 0}$.

Summarizing the results of the two cases leads to the condition
\begin{align}
\max\calL_{\calF}V(x)\leq0
\end{align}
for all $x\in\R^{kn}$, which proves strong invariance of $\calS_p(C)$ for all $C>0$.

\textit{(2)}. Let $p=\infty$. Consider 
\begin{align}
V(x) = \max_{i\in\calI} \|x_i\|_\infty
\label{eqn_prf_lm:invariant_V_inf}
\end{align}
as a Lyapunov function candidate. Since the proof shares the same structure and reasoning as the case $p\in[1,\infty)$, we only provide a sketch of the proof. 

In this case, the set $\alpha(x)$ in \eqref{eqn_prf_lm:invariant_alpha} is 
\begin{align}
\alpha(x) = \big\{ i\in\calI \mid \|x_i\|_\infty = V(x) \big\},
\end{align}
whereas the generalized gradient of $V$ reads
\begin{align}
\partial V(x)= &\mathrm{co} \big\{ e_i\otimes \big(\calF[\sign](x_{i,\ell})e_\ell\big) \mid e_i\in\R^n, \nonumber\\
&  e_\ell\in\R^k, |x_{i,\ell}|=V(x) \big\}.
\end{align}

For the case $x\in\calC$, we have $\max\calL_{\calF[h]}V(x(t)) = 0 \leq 0$ by using the same argument as the case $p\in[1,\infty)$. Hence, we omit the details.

For the case $x\notin\calC$, we first show that there exists an index $i\in\alpha(x)$ such that $\bar{L}_ix \neq \mathbf{0}$ by using contradiction. If $\alpha(x)=\calI$, the conclusion follows as the case \textit{(1)}. If $\alpha(x)\subsetneq \calI$ and $\bar{L}_ix = \mathbf{0}$ for all $i\in\alpha(x)$, then for any $i\in\alpha(x)$ with $j\in N_i\setminus\alpha(x)$ we have
\begin{align}
0 & = \big(\calF[\sign](x_{i,\ell})e_\ell\big)^\top\bar{L}_ix  \label{eqn_prf_lm:invariant_Lix_infty-2}\\
& = \sum_{j\in N_i} \bigg( |x_{i,\ell}| - \calF[\sign](x_{i,\ell})x_{j,\ell} \bigg)\\
& \geq \sum_{j\in N_i} \bigg(\|x_i\|_\infty - \|x_{j}\|_\infty \bigg) \label{eqn_prf_lm:invariant_Lix_infty}\\
& > 0
\end{align}
where $e_i\otimes \big(\calF[\sign](x_{i,\ell})e_\ell\big)\in\partial V(x)$ and the last inequality is implied by $\|x_i\|>\|x_j\|$ for any $j\in N_i\setminus\alpha(x)$. This is a contradiction.

For the index $i\in\alpha(x)$ satisfying $\bar{L}_ix\neq \mathbf{0}$, it follows from Assumption~\ref{ass_vectorfield_1} that there exists $\gamma>0$ such that for any $\nu\in\bar{\calF}(x)$ it holds that $\nu_i = -\gamma\bar{L}_ix$. Using the same reasoning as in case \textit{(1)}, by choosing $\zeta\in\partial V(x)$ as $\zeta = e_i\otimes \big(\calF[\sign](x_{i,\ell})e_\ell\big)$ (recall that $i\in\alpha(x)$), it follows from (\ref{e:set-valuedLie}) that
\begin{align}
\calL_{\bar{\calF}[h]}V(x) = \{-\gamma \big(\calF[\sign](x_{i,\ell})e_\ell\big)^\top\bar{L}_ix\}.
\end{align}
Next, by observing \eqref{eqn_prf_lm:invariant_Lix_infty} we have
$(\calF[\sign](x_{i,\ell})e_\ell)^\top\bar{L}_ix \geq 0,$ 
which implies $\calL_{\bar{\calF}[h]}V(x)\subset \R_{\leq 0}$.

In summary, we have
\begin{align}
\max\calL_{\calF}V(x)\leq0
\end{align}
for all $x\in\R^{kn}$, which proves strong invariance of $\calS_p(C)$ for all $C>0$.
\end{proof}

Now we are in the position to state the main result of this paper.
\begin{thm}\label{th:main_vector}
	Consider the nonlinear consensus protocol (\ref{e:nonlinear_vector}) satisfying Assumption~\ref{ass_vectorfield_1} and the corresponding differential inclusion (\ref{e:nonlinear_vector_fili}) for some $p\in[0,\infty]$. Then, the following statements hold:
	\begin{itemize}
		\item[\textit{(i)}] If $|\calI| > 2$, then all Filippov solutions of (\ref{e:nonlinear_vector_fili}) converge to consensus in finite time if and only if $|\calI_c| = 1$;
		\item[\textit{(ii)}] If $|\calI| = 2$, then all Filippov solutions of (\ref{e:nonlinear_vector_fili}) converge to consensus in finite time if and only if $|\calI_c| \leq 1$.
	\end{itemize}
\end{thm}

\begin{proof}
The proofs of sufficiency and necessity of the two statements are considered separately.

\emph{Sufficiency.} The Lyapunov function candidate
\begin{align}
V(x) = \sqrt{x^\top\bar{L}x}
\label{eqn_prf_tm:main_vector_suff_V}
\end{align}
is introduced. Note that $V(x) = 0$ for all $x\in\calC$ and that $V$ is convex, hence, regular. The set-valued Lie derivative of $V$ in (\ref{eqn_prf_tm:main_vector_suff_V}) will be considered for $x\notin\calC$, hereby evaluating the two cases in the statement of Theorem~\ref{th:main_vector} separately. In both cases, it will be shown that
\begin{align}
\max\calL_{\calF[h]}V(x)\leq -c
\end{align}
for some $c>0$ that is independent of $x\notin\calC$. Then, finite-time consensus follows from Lemma~\ref{p:finite_time_convergence}, hereby exploiting a strongly invariant set from Lemma~\ref{lm:invariant_vec}.

\textit{(i)}. The case $|\calI>2|$ is considered first, and we assume that $|\calI_c| = 1$.
As in the proof of Lemma~\ref{lm:invariant_vec}, we use the extended differential inclusion $\bar{\calF}$ defined in (\ref{eqn_prf_lm:invariant_Finclusion}), which satisfies the property (\ref{eqn_prf_lm:invariant_LVinclusion}). Consequently, $\calL_{\calF[h]}V(x) \subset \calL_{\bar{\calF}}V(x)$ with the latter given by (\ref{e:set-valuedLie}) as
\begin{align}
\!\calL_{\bar{\calF}}V(x) = \left\{ \frac{1}{\sqrt{x^\top\bar{L}x}}x^\top\bar{L}\nu \bigg|\; \nu \in \bar{\calF}(x) \right\}\!,\!
\label{eqn_prf_tm:main_vector_suff_LV}
\end{align}
which follows from the observation that the generalized gradient of (\ref{eqn_prf_tm:main_vector_suff_V}) reduces to the regular gradient for \mbox{$x\notin\calC$}.

As the case $|\calI| > 2$, $|\calI_c| = 1$ is considered, there exists exactly one agent with locally Lipschitz continuous dynamics. Without loss of generality, let this agent have index $1$. Then, it follows from the property of $f_1$ in Assumption~\ref{ass_vectorfield_1} that
\begin{align}
x^\top\bar{L}^\top_1f_1\big(-\bar{L}_1x\big) \leq 0.
\label{eqn_prf_tm:main_vector_suff_LVcont}
\end{align}
Next, for $i\in\{2,\ldots,n\}$, it holds that
\begin{align}
\calF[\sign]\big(-\bar{L}_ix\big) = \left\{\begin{array}{ll}
\left\{ \frac{-\bar{L}_ix}{\|\bar{L}_ix\|_p} \right\}, & \|\bar{L}_ix\| \neq 0, \\
\{v\mid \|v\|_p\leq 1\}, & \|\bar{L}_ix\| = 0,
\end{array}\right.
\end{align}
such that, for all $\nu_i\in\calF[\sign](-\bar{L}_ix)$,
\begin{align}
x^\top\bar{L}^\top_i\nu_i &= \left\{\begin{array}{ll}
\frac{-\|\bar{L}_ix\|_2^2}{\|\bar{L}_ix\|_p}, & \|\bar{L}_ix\| \neq 0, \\
0, & \|\bar{L}_ix\| = 0.
\end{array}\right. \label{eqn_prf_tm:main_vector_suff_LVsign_step1}
\end{align}
Due to the equivalence of norms on finite-dimensional vector spaces, there exists $d_1>0$ such that $d_1\|\bar{L}_ix\|_p \leq \|\bar{L_i}x\|_2$ for all $x\in\R^{nk}$. Applying this to (\ref{eqn_prf_tm:main_vector_suff_LVsign_step1}) yields
\begin{align}
x^\top\bar{L}^\top_i\nu_i \leq -d_1\|\bar{L}_ix\|_2
\label{eqn_prf_tm:main_vector_suff_LVsign_step2}
\end{align}
for $i\in\{2,\ldots,n\}$ and note that this indeed holds for both cases in (\ref{eqn_prf_tm:main_vector_suff_LVsign_step1}).
Now, after recalling the definition  of $\calL_{\bar{\calF}}V(x)$ in (\ref{eqn_prf_tm:main_vector_suff_LV}), the combination of (\ref{eqn_prf_tm:main_vector_suff_LVcont}) and (\ref{eqn_prf_tm:main_vector_suff_LVsign_step2}) shows that, for any $a\in\calL_{\bar{\calF}}V(x)$,
\begin{align}
a \leq -\frac{d_1}{\sqrt{x^\top\bar{L}x}} \left( \sum_{i=2}^n \|\bar{L}_ix\|_2 \right).
\label{eqn_prf_tm:main_vector_suff_abound_step1}
\end{align}
%
Note that, as $\bar{L} = L\otimes I_k$  with $L$ a graph Laplacian satisfying $\one ^\top L = \mathbf{0}^\top$, it holds that
\begin{align}
\|\bar{L}_1x\| = \left\| \sum_{i=2}^n \bar{L}_ix \right\| \leq \sum_{i=2}^n \big\| \bar{L}_ix \big\|,
\label{eqn_prf_tm:main_vector_suff_L1bound}
\end{align}
where the triangle inequality is used to obtain the inequality. Then, the use of (\ref{eqn_prf_tm:main_vector_suff_L1bound}) in (\ref{eqn_prf_tm:main_vector_suff_abound_step1}) yields
\begin{align}
a \leq -\frac{1}{2}\frac{d_1}{\sqrt{x^\top\bar{L}x}} \left( \sum_{i=1}^n \|L_ix\| \right).
\label{eqn_prf_tm:main_vector_suff_abound_step2}
\end{align}
By further exploiting that $L$ is a graph Laplacian, it holds that $L$ and $L^\top L$ can be written as
\begin{align}
L = U^\top\Lambda U, \quad L^\top L = U^\top\Lambda^2 U,
\label{eqn_prf_tm:main_vector_suff_Ldecomposition}
\end{align}
where $\Lambda = \diag\{0,\lambda_2,\ldots,\lambda_n\}$ is a diagonal matrix with Laplacian real-valued eigenvalues satisfying $0<\lambda_2$ and $\lambda_{j}\leq\lambda_{j+1}$, and the matrix $U$ collects the corresponding eigenvectors. From (\ref{eqn_prf_tm:main_vector_suff_Ldecomposition}), it can be seen that
\begin{align}
L^\top L - c_1L \succcurlyeq 0
\end{align}
for any $c_1\in(0,\lambda_2]$. Consequently, using $\bar{L} = L\otimes I_k$, it follows that
\begin{align}
\left( \sum_{i=1}^n \|L_ix\| \right)^2 = x^\top L^\top Lx \geq c_1 x^\top Lx.
\label{eqn_prf_tm:main_vector_suff_Lbound}
\end{align}
After taking the square root (note that $x^\top L x>0$ for all $x\notin\calC$) in (\ref{eqn_prf_tm:main_vector_suff_Lbound}) and using (\ref{eqn_prf_tm:main_vector_suff_abound_step2}), the result
\begin{align}
a \leq -\frac{\sqrt{c_1}}{2}\frac{d_1\sqrt{x^\top\bar{L}x}}{\sqrt{x^\top\bar{L}x}} = -\frac{d_1\sqrt{c_1}}{2}<0
\end{align}
follows.  

\textit{(ii)}. The proof for the case $|\calI| = 2$ and $|\calI_c|\leq1$ follows similarly.

Next by Lemma~\ref{p:finite_time_convergence}, we have that the trajectories converge to $Z_{\bar{\calF},V}$ in finite time. The remaining task to characterize the set $Z_{\bar{\calF},V}$. So far we have shown that $x\notin Z_{\bar{\calF},V}$ for $\forall x\notin \calC$ which implies that $Z_{\bar{\calF},V}\subset\calC$. By the fact that $\calC$ is closed, we have $\overline{Z_{\bar{\calF},V}}\subset \calC$. Moreover when $x\in\calC$, $\dot{x}_i=0$ where $\{i\}=\calI_c$ which implies $x_i$ remains constant. In conclusion, the finite-time convergence to static consensus is guaranteed.

	\emph{Necessity:} The necessity of the conditions in  \textit{(i)} and \textit{(ii)} can be proven by showing the following equivalent formulation: For any $|\calI|\geq2$ and $|\calI_c|\geq2$, there exists (at least one) Filippov solution of (\ref{e:nonlinear_vector_fili}) that does not converge to (static) consensus in finite time. Note that the conditions in both statements are now considered simultaneously.
	
	For the case $\calI=\calI_c$, the desired result immediately follows from Proposition~\ref{prop: asymptotic_convergence}. Therefore, in the remainder of this proof, we consider the case in which $\calI_c$ is a strict subset of $\calI$ and we  restrict analysis to the case $|\calI| \geq 3$.
	
	Consider the function 
	\begin{equation}\label{eqn_Vi}
	V_i(x)=\|\bar{L}_i x\|_p
	\end{equation} for $i\in\calI\setminus\calI_c$ and the set
	\begin{equation}\label{eqn_setSdelta}
	S(\delta) = \big\{ x\in\R^{kn} \;\big|\; \sqrt{x^\T\bar{L}x} \leq \delta \big\}.
	\end{equation}
	Note that $\calC\subset S(\delta)$ for any $\delta\leq 0$. By the sufficiency proof of Theorem~\ref{th:main_vector}, we have that $S(\delta)$ is strongly invariant. More precisely, by the Lipschitz continuity of $f_i$ for $i\in\calI_c$, there exists $\varepsilon$ such that  $\|f_i(-\bar{L}_ix)\|_2\leq 1$ for any $\|\bar{L}_ix\|_2\leq \varepsilon$ and $i\in\calI_c$. Furthermore, 
	\begin{align}
	\bigg\|\sum_{j\in N_i} (x_j-x_i)\bigg\|_2 & \leq \sum_{j\in N_i} \|x_j-x_i\|_2 \\
	&\leq \bigg(m (\sum_{(i,j)\in \calE} \|x_j-x_i\|_2^2) \bigg)^{\frac{1}{2}},
	\end{align}
	where we recall that $m=|\calE|$. Then, by choosing $\delta_\varepsilon<\frac{\varepsilon}{\sqrt{m}}$, we have $\|\bar{L}_ix\|_2\leq \varepsilon$ for any $i\in\calI_c$ and $x\in S(\delta_\varepsilon)$.	By the equivalence of the $\ell_p$-norms on finite dimensional space, we assume without loss of generality that we can choose proper $\delta_\varepsilon$ such that $\|\bar{L}_ix\|_p\leq \varepsilon$ for any $i\in\calI_c$ and $x\in S(\delta_\varepsilon)$.
	
	We consider the evolution of $V_i$ along trajectories $x$ in $S(\delta_\varepsilon )$. To this end, note that $V_i(x)$ in (\ref{eqn_Vi}) is locally Lipschitz and convex and, as a result, regular. In the evaluation of the set-valued Lie derivative $\calL_{\calF[h]}V_i(x)$ as in (\ref{e:set-valuedLie}), we will only consider the subset of time for which $\calL_{\calF[h]}V_i(x(t))$ is non-empty. Moreover, as before, the cases $p\in[1,\infty)$ and $p=\infty$ will be considered separately. We denote $\zeta=[\zeta_1,\ldots,\zeta_n]^\top$ and $\nu=[\nu_1,\ldots,\nu_n]^\top$, where $\zeta_j, \nu_j\in\R^k$, $j=1,\ldots,n$, for any $\zeta\in \partial V_i(x)$ and $\nu\in\calF[h](x)$, respectively.
	
	\textit{(1)}. Let $p\in[1,\infty)$.
	If $V_i(x)\neq 0$, then we have that 
	\begin{align}
	\zeta_{j} & = \mathbf{0},\quad  & j\notin N_i \\
	\zeta_{i} & \in L_{ii} \|\bar{L}_i x \|_p^{1-p}\psi(\bar{L}_i x)  \\
	\zeta_{j} & \in \|\bar{L}_i x \|_p^{1-p}\psi(\bar{L}_i x), \quad  & j\in N_i.
	\end{align}
	where $\psi$ is given as in \eqref{e:psi}. If $V_i(x)= 0$, we have that $\zeta$ satisfies $\zeta_{j}=\mathbf{0}$ for $j\notin N_i$ and the vector $\omega$, whose components are $\frac{1}{L_{ii}}\zeta_{i}$ and $\zeta_j$ for all $j\in N_i$, satisfies $\|\omega\|_q\leq 1$ with $q=\frac{p}{p-1}$ for $p>1$ and $q=\infty$ for $p=1$.
	
	If $V_i(x)=0$, the Lie derivative $\calL_{\calF[h]}V_i(x)=\{0\}$. Indeed, if there exists $a\in\calL_{\calF[h]}V_i(x)$, there exists $\nu\in\calF[h](x)$ such that $a=\nu_i^\top\zeta_i+\sum_{j\in N_i}\nu_j^\top\zeta_j$ for any $\omega$, where the components of $\omega$ are $\frac{1}{L_{ii}}\zeta_{i}$ and $\zeta_j$ for all $j\in N_i$, satisfying $\|\omega\|_{q}\leq 1$. Hence such $a$ can only be equal to $0$ and $\max\calL_{\calF[h]}V_i(x)\leq 0$.
	
	If $V_i(x)\neq 0$, for any $a\in\calL_{\calF[h]}V_i(x)$, there exists a $\nu\in\calF[h](x)$ such that
	\begin{equation}
	\begin{aligned}
	a & = \sum_{j=1}^n \|\bar{L}_i x \|_p^{1-p} L_{ij}\nu_j^\top  \psi(\bar{L}_i x) \\
	& \leq \sum_{j=1, j\neq i}^n |L_{ij}|-L_{ii} \\
	& = 0,
	\end{aligned}
	\end{equation}
	where the inequality is implied by H\"{o}lder's inequality and the fact that $\|\nu_j\|_p\leq 1, j\neq i$, $\nu_i=-\frac{\bar{L}_i x}{\|\bar{L}_i x \|_p}$ and $\|\psi(\bar{L}_i x)\|_q=\|\bar{L}_i x \|_p^{p-1}$.
	
	\textit{(2)}. Let $p=\infty$. The generalized gradient of $V_i$ is given as 
	\begin{align}
	\partial V_i(x)= \mathrm{co}\{\zeta \mid \zeta_i &= e_\ell \calF[\sign]((\bar{L}_ix)_\ell)L_{ii}, \nonumber\\ 
	\zeta_j & = - e_\ell \calF[\sign]((\bar{L}_ix)_\ell), \nonumber\\
	e_\ell & \in\R^k, j\in N_i, |(\bar{L}_ix)_\ell|=V_i(x)  \}.
	\end{align}
	Then for any $a\in\calL_{\calF[h]}V_i(x)$, there exists a $\nu\in\calF[h](x)$ such that
		\begin{align}
		a & = \nu_i^\top\zeta_i+\sum_{j\in N_i}\nu_j^\top\zeta_j \nonumber\\ 
		& = \calF[\sign]((\bar{L}_ix)_\ell)L_{ii}e^\top_\ell \nu_i-  \sum_{j\in N_i}\calF[\sign]((\bar{L}_ix)_\ell)e^\top_\ell \nu_j \nonumber\\
		& \leq \sum_{j=1, j\neq i}^n |L_{ij}|-L_{ii} \nonumber\\
		& = 0. \label{e:a-pinfty}
		\end{align}
	
	Combining the above two cases, we have shown that
	\begin{align}
	\max \calL_{\calF[h]}V_i(x(t)) \leq 0
	\label{eqn_Liederivative_Vi}
	\end{align}
	for any $i\in\calI\setminus\calI_c$ and for $x(t)\in S(\delta_\varepsilon)$.
	
	In the remainder of this proof, we will construct a Filippov solution of (\ref{e:nonlinear_vector_fili}) in $S(\delta_\varepsilon)$ which does not achieve static consensus in finite time. Here, we assume without loss of generality that the nodes are ordered such that $\calI_c = \{1,\ldots,|\calI_c|\}$ and $\calI\setminus\calI_c = \{|\calI_c|+1,\ldots,n\}$ and partition the Laplacian $L$ accordingly as
	\begin{align}
	L=\begin{bmatrix}
	L_{cc} & L_{cd} \\
	L_{dc} & L_{dd}
	\end{bmatrix}.
	\end{align}
	Then, consider the solutions for an initial condition $x_0 = [ \, x_0^c \;\; x_0^d \, ]^{\T}\in S(\delta)$ that satisfies the set of equations
	\begin{align}
	\bar{L}_{cc}x_0^c + \bar{L}_{cd}x_0^d &\neq \mathbf{0}, \label{eqn_Lpart_cont} \\
	\bar{L}_{dc}x_0^c + \bar{L}_{dd}x_0^d &= \mathbf{0}, \label{eqn_Lpart_discont}
	\end{align}
	and note that such solution exists as $\bar{L}_{dd}$ is invertible. (This follows from the standing assumption that the graph $\calG$ is connected and Lemma~\ref{lem_schur complement}.) Recall that the left-hand side of (\ref{eqn_Lpart_discont}) can be written as $\bar{L}_ix_0 = \mathbf{0}$ for $i\in\calI\setminus\calI_c$ and, thus, as $V_i(x_0) = 0$ with $V_i$ as in (\ref{eqn_Vi}). Furthermore, since $|\calI_c|\geq 2$, the solution $x_0$ can be chosen such that $x^c_0\notin\calC$. Then, by the result (\ref{eqn_Liederivative_Vi}), it follows that $V_i$ is non-increasing along trajectories, such that $V_i(x(t)) = 0$ for all $i\in\calI\setminus\calI_c$ and all $t\geq 0$. Consequently, any trajectory with initial condition $x_0\in S(\delta)$ satisfying (\ref{eqn_Lpart_cont})--(\ref{eqn_Lpart_discont}) satisfies $x^d(t) = -\bar{L}_{dd}^{-1}\bar{L}_{dc}x^c(t)$ for all $t\geq 0$.
	
	In this case, the dynamics of the nodes with continuous dynamics can be expressed as
	\begin{align}
	\dot{x}^c = f^c\big( -(\bar{L}_{cc} - \bar{L}_{cd}\bar{L}_{dd}^{-1}\bar{L}_{dc})x^c \big),
	\label{eqn_xcdynamics}
	\end{align}
	where $f^c$ collects the Lipschitz continuous functions $f_i$ in Assumption~\ref{ass_vectorfield_1} for $i\in\calI_c$. As a result of Lemma~\ref{lem_schur complement}, the matrix $L_{cc} - L_{cd}L_{dd}^{-1}L_{dc}$ is itself a graph Laplacian, such that the dynamics (\ref{eqn_xcdynamics}) can be regarded as a special case of (\ref{e:nonlinear_vector}) in which all nodes have continuous dynamics. As such, the result follows from Proposition~\ref{prop: asymptotic_convergence}, finalizing the proof of necessity.	
\end{proof}

\begin{rem}
In Lemma \ref{lm:invariant_vec} and Theorem \ref{th:main_vector}, we set the edge weights of the graph $\calG$, i.e., $w_{ij}$, to one to simplify the notation in the proofs. However, all results in this paper hold for general positive $w_{ij}$. For example, it can be verified that the calculations in \eqref{eqn_prf_lm:invariant_Lix-4}-\eqref{eqn_prf_lm:invariant_Lix}, \eqref{eqn_prf_lm:invariant_Lix_infty-2}-\eqref{eqn_prf_lm:invariant_Lix_infty} and \eqref{e:a-pinfty} hold for the case with general edge weights.
\end{rem}

We close this subsection with demonstrating the result in Theorem \ref{th:main_vector} by an example.

\begin{exmp}\label{ex:direction-preserving signum}
Consider the system \eqref{e:nonlinear_vector_fili} with $x_i\in\R^2$ defined on the graph given in Fig.~\ref{fig:ex_topology} and satisfying Assumption \ref{ass_vectorfield_1}. Let $p=2$ (i.e., the $\ell_2$-norm is considered) in $\sign$ defined by \eqref{e:sign}.

First, consider $f_i=\sign, i=1,\ldots,4$ and $f_5$ is the identity function. Hence condition $(i)$ in Theorem~\ref{th:main_vector} is satisfied. The trajectory of this system with randomly generated initial conditions is depicted in Fig.~\ref{fig:ex_FTC}. Here we can see that finite-time consensus is achieved.

Next, using the same initial conditions, set $f_i=\sign, i=1,\ldots,3$ and $f_4,f_5$ the identity function. Then both conditions $(i)$ and $(ii)$ are violated, so finite-time consensus is not expected. Indeed, in this case we can only have asymptotic convergence to consensus as shown in Fig.~\ref{fig:ex_AC}.

Notice that in both cases, the trajectories do not exceed the unit circle which is compatible with Lemma~\ref{lm:invariant_vec}.
\end{exmp}

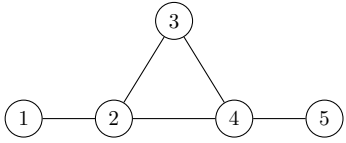
\begin{figure}
\centering
\begin{tikzpicture}
    \node[shape=circle,draw=black,scale=0.8] (A) at (-0.5,0) {1};
    \node[shape=circle,draw=black,scale=0.8] (B) at (0.7,0) {2};
    \node[shape=circle,draw=black,scale=0.8] (C) at (1.5,1.3) {3};
    \node[shape=circle,draw=black,scale=0.8] (D) at (2.3,0) {4};
    \node[shape=circle,draw=black,scale=0.8] (E) at (3.5,0) {5};

    \path [-](A) edge (B);
    \path [-](B) edge (C);
    \path [-](B) edge (D);
    \path [-](D) edge (C);
    \path [-](D) edge (E);  
\end{tikzpicture}
\caption{The underlying topology for the system in Example~\ref{ex:direction-preserving signum}}\label{fig:ex_topology}
\end{figure}

\begin{figure}[t!]
\centering
\begin{subfigure}[t]{0.4\textwidth}
\includegraphics[width=1\textwidth]{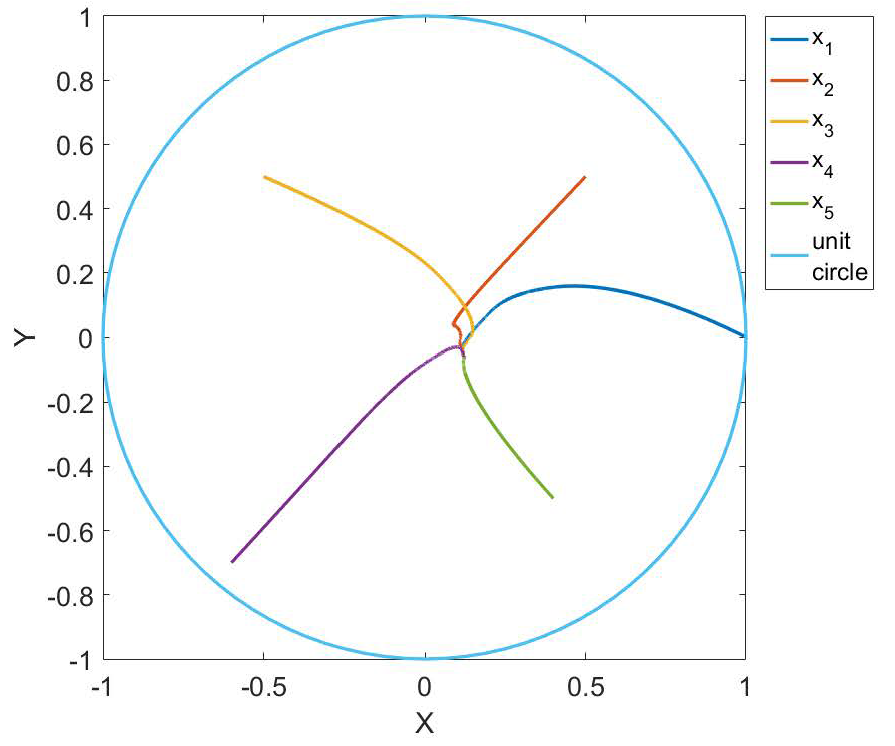}
\caption{Phase portraits of $x_i,i=1\ldots,5$ in the unit disc in $\R^2$.}\label{fig:ex_FTC_div1}
\end{subfigure}
~
\begin{subfigure}[t]{0.4\textwidth}
\includegraphics[width=1\textwidth]{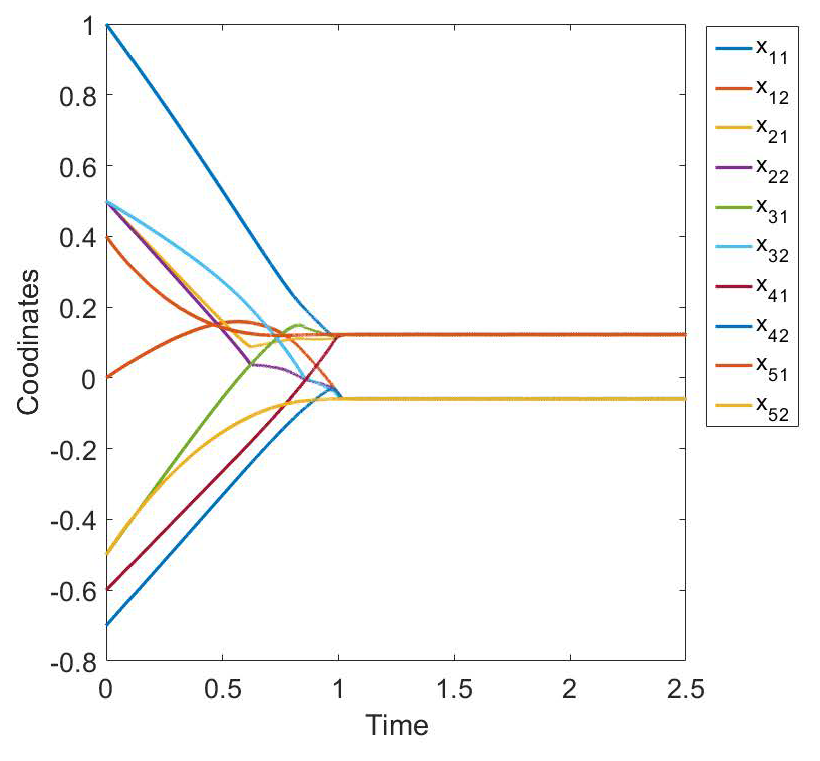}
\caption{Evolution of the two coordinates of $x_i,i=1\ldots,5$. Finite-time consensus is achieved.}\label{fig:ex_FTC_div2}
\end{subfigure}
\caption{Simulation of the first scenario in Example \ref{ex:direction-preserving signum}.}\label{fig:ex_FTC}
\end{figure}

\begin{figure}[t!]
\centering
\begin{subfigure}[t]{0.4\textwidth}
\includegraphics[width=1\textwidth]{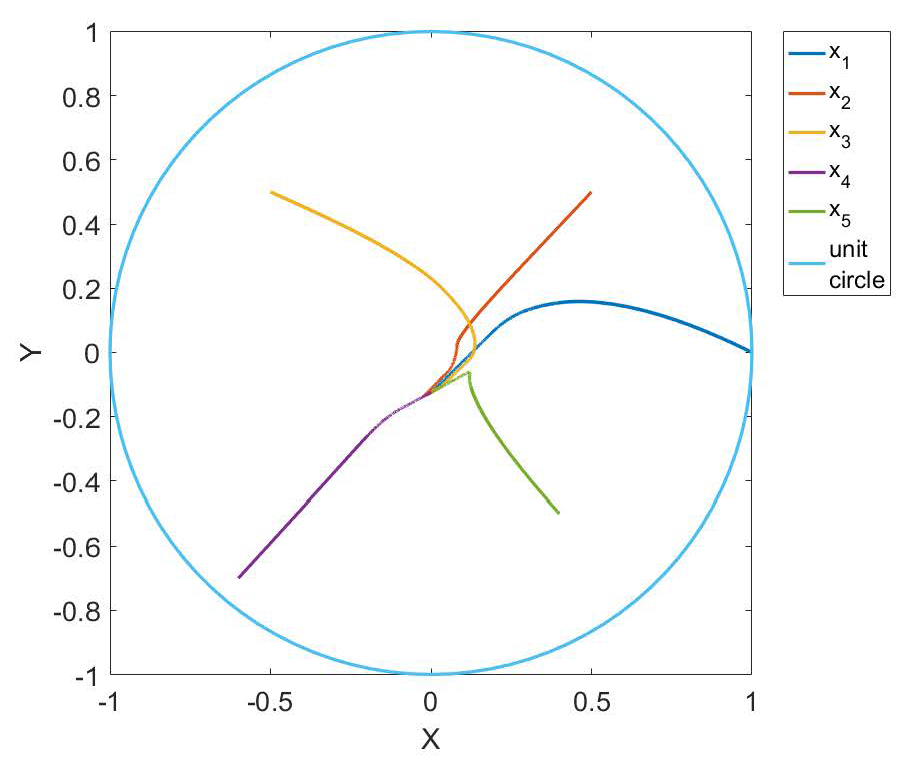}
\caption{Phase portraits of $x_i,i=1\ldots,5$ in the unit disc in $\R^2$.}\label{fig:ex_AC_div1}
\end{subfigure}
~
\begin{subfigure}[t]{0.4\textwidth}
\includegraphics[width=1\textwidth]{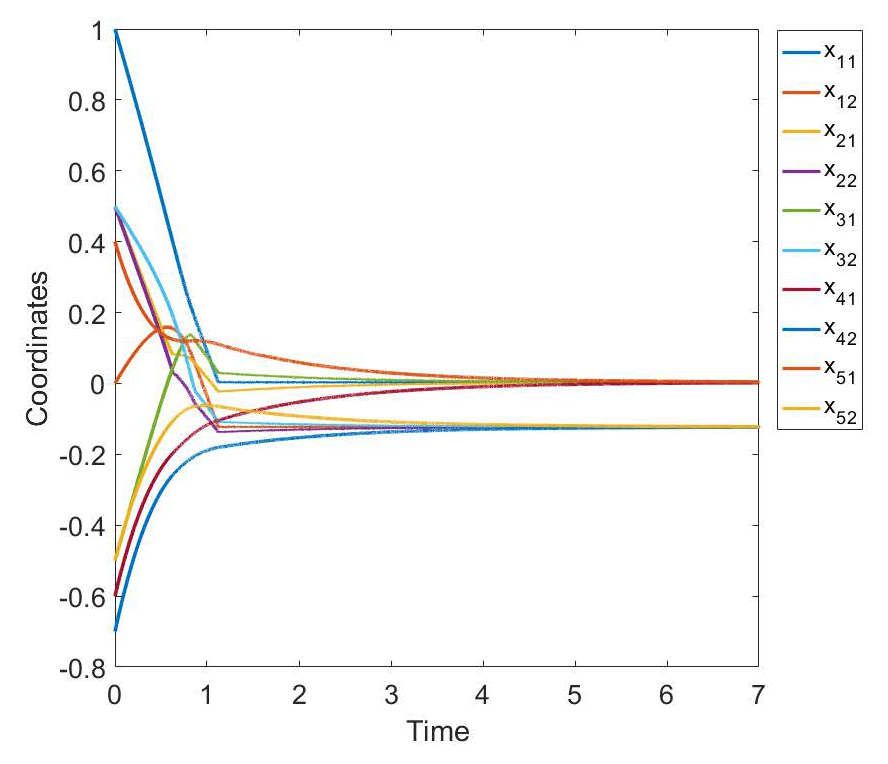}
\caption{Time evolution of the two coordinates of $x_i,i=1\ldots,5$. Here the convergence to consensus is asymptotic but not in finite time.}\label{fig:ex_AC_div2}
\end{subfigure}
\caption{Simulation of the second scenario in Example \ref{ex:direction-preserving signum}.}\label{fig:ex_AC}
\end{figure}

\subsection{Component-wise signum}\label{s:component sign}

In this subsection, we study another finite-time consensus controllers in the form \eqref{e:control} using the component-wise signum function \eqref{e:signc}. Similar to the reasoning in subsection \ref{s:direction-preserving sign},  in order to achieve the convergence to a static vector in $\calC$, we introduce the following assumption on the nonlinear functions $f_i$ in \eqref{e:control}. 

\begin{assum}\label{ass_vectorfield_2}
For some set $\calI_c\subset\calI$, the function $f$ in~(\ref{e:nonlinear_vector}) satisfies the following conditions:
	\begin{enumerate}
		\item[\textit{(i)}] For $i\in\calI_c$, the function $f_i:\R^k\rightarrow\R^k$ satisfies $f_i(y)=[f_{i,1}(y_1),\ldots,f_{i,k}(y_k)]^\top$, where $f_{i,j}$ is locally Lipschitz continuous satisfying $f_{i,j}(0) = 0$ and $f_{i,j}(y_j) y_j > 0$ for all $y_j\neq 0$ and $j=1,\ldots,k$;
		\item[\textit{(ii)}] For $i\in\calI\backslash\calI_c$, the function $f_i = \sign_c$.
	\end{enumerate}
\end{assum}

Based on Assumption \ref{ass_vectorfield_2} and the component-wise signum function, the dynamics of the components are decoupled. Consequently,  the following results can directly be derived from Section \ref{s:direction-preserving sign}.

\begin{cor}
Consider the differential inclusion (\ref{e:nonlinear_vector_fili}) satisfying Assumption \ref{ass_vectorfield_2}. If one of the following two conditions is satisfied, 
\begin{enumerate}
\item[\textit{(i)}] $|\calI| = 2$ and $|\calI_c| = 0$;
\item[\textit{(ii)}] $|\calI| \geq 2$ and $|\calI_c| \geq 1$,
\end{enumerate}
then the set $\calS_\infty(C) = \{ x\in\R^{nk}\mid \|x_i\|_\infty\leq C, i\in\calI \}$, where $C>0$ is a constant, is strongly invariant.
\end{cor}

Notice that the controller \eqref{e:control} satisfying Assumption \ref{ass_vectorfield_2} has the invariant set $\calS_\infty$ defined by the $\ell_\infty$-norm. This is one major difference compared to \eqref{e:control} satisfying Assumption \ref{ass_vectorfield_1}. The following corollary to Theorem \ref{th:main_vector} is obtained in the scope of component-wise signum functions. 

\begin{cor}
Consider the nonlinear consensus protocol (\ref{e:nonlinear_vector}) satisfying Assumption~\ref{ass_vectorfield_2} and the corresponding differential inclusion (\ref{e:nonlinear_vector_fili}). Then, the following statements hold:
\begin{itemize}
		\item[\textit{(i)}] If $|\calI| > 2$, then all Filippov solutions of (\ref{e:nonlinear_vector_fili}) converge to consensus in finite time if and only if $|\calI_c| = 1$;
		\item[\textit{(ii)}] If $|\calI| = 2$, then all Filippov solutions of (\ref{e:nonlinear_vector_fili}) converge to consensus in finite time if and only if $|\calI_c| \leq 1$.
\end{itemize}
\end{cor}

\section{Conclusion}\label{s:conclusions}
In this paper, we considered the finite-time consensus problem for high-dimensional multi-agent systems. Two finite-time consensus control protocols are proposed, one using direction-preserving signum and another using component-wise signum. The second controller uses coarser information compared to the first one. However the second one can only guarantee that the unit ball in $\ell_\infty$-norm is strongly invariant, while the first one can be designed such that the $\ell_p$-norm unit ball is strongly invariant for any $p\in[1,\infty]$. For these two controller, sufficient and necessary conditions were presented to guarantee finite-time convergence and boundedness of the solutions.

\bibliographystyle{plain} 
\bibliography{ref}

\end{document}